\documentclass[twoside,12pt]{amsart}
\usepackage{amsmath,amsfonts}
\usepackage[dvips]{graphicx}
\usepackage{color}
\usepackage{amsthm}
\usepackage{amsfonts}
\usepackage{amssymb}
\usepackage{latexsym}
\numberwithin{equation}{section}

\makeatletter
\def\serieslogo@{}
\makeatother \makeatletter
\def\@setcopyright{}
\makeatother


\usepackage{color}

\def\p{\partial}
 
 \def\R{{\mathbb R}}
 
 \def\dis{\displaystyle}

 \def\ol{\overline}

 \def\bf{\textbf}
 \def\T{\mathcal{T}}
 \def\F{\mathcal{F}}
 \def\A{\mathcal{A}}

\def\oc{\overset\circ}
\def\Hb{{\oc H{}^1_\beta}}
\def\Hg{{\oc H{}^1_\gamma}}
\def\df{\overset{\text{def.}}=}

\newtheorem{thm}{Theorem}
\newtheorem{lem}[thm]{Lemma}

\newtheorem{prop}[thm]{Proposition}

\newtheorem{defn}[thm]{Definition}

\title[Global well-posedness for the microscopic FENE model]{Global well-posedness for the microscopic FENE model with a sharp boundary condition}

\author{
Hailiang Liu and Jaemin Shin
}
\address{Iowa State University, Mathematics Department, Ames, IA 50011} \email{hliu@iastate.edu}
\address{Iowa State University, Mathematics Department, Ames, IA 50011} \email{jmshin@iastate.edu}

\date{November 26, 2009}
\begin{document}

\bibliographystyle{abbrv}

\begin{abstract} We prove global well-posedness for the microscopic FENE model
under a sharp boundary requirement. The well-posedness of the FENE
model that consists of the incompressible Navier-Stokes equation and
the Fokker-Planck equation has been studied intensively, mostly with
the zero flux boundary condition. Recently it was illustrated by C.
Liu and H. Liu [2008, SIAM J. Appl. Math., 68(5):1304--1315] that
any preassigned boundary value of a weighted distribution will
become redundant once the non-dimensional parameter $b>2$. In this
article, we show that for the well-posedness of the  microscopic
FENE model ($b>2$) the least boundary requirement is that the
distribution near boundary needs to approach zero faster than the
distance function. Under this condition, it is shown that there
exists a unique weak solution in a weighted Sobolev space. Moreover, such a condition still ensures that the
distribution is a probability density. The sharpness of this boundary requirement is shown by a construction of
infinitely many solutions when the distribution approaches zero as
fast as the distance function.
\end{abstract}

\maketitle

\bigskip

\tableofcontents

\section{Introduction}
It is well-known that the following system coupling incompressible Navier-Stokes equation for the macroscopic velocity field $v(t,x)$ and the Fokker-Planck equation for the probability density function $f(t,x,m)$ describes diluted solutions of
polymeric liquids with noninteracting polymer chains
 \begin{eqnarray}
\p_t v + (v\cdot \nabla) v + \nabla p &=& \nabla \cdot \tau_p + \nu \Delta v, \label{n1}\\
\nabla\cdot v&=&0,\label{n2}\\
\p_tf+ (v\cdot \nabla)f+ \nabla_m\cdot(\nabla v m f)&=&\frac2\zeta \nabla_m \cdot(\nabla_m \Psi(m)f)+\frac{2kT}{\zeta}\Delta_m f,\label{n3}
\end{eqnarray}
where $x\in\R^n$ is the macroscopic Eulerian coordinate and $m\in \R^n$ is the microscopic molecular configuration variable. In this model, a polymer is idealized as an elastic dumbbell consisting of two beads joined by a spring that can be modeled by a vector $m$ (see e.g \cite{BCAH87}).
In the Navier-Stokes equation \eqref{n1}, $p$ is hydrostatic pressure, $\nu$ is the kinematic viscosity coefficient, and $\tau_p$ is a tensor representing the polymer contribution to stress,
\begin{equation*}
\tau_p=\lambda \int m\otimes\nabla_m \Psi(m)f dm,
\end{equation*}
where $\Psi$ is the elastic spring potential and $\lambda$ is the
polymer density constant. In the Fokker-Planck equation \eqref{n3},
$\zeta$ is the friction coefficient of the dumbbell beads, $T$ is
the absolute temperature, and $k$ is the Boltzmann constant. Notice
that the Fokker-Planck equation can be written as a stochastic
differential equation (see \cite{Oe:1996}).

One of the simplest model is the Hookean model in which the potential $\Psi$ is given by
\begin{equation*}
\Psi(m)=\frac{H |m|^2}{2},
\end{equation*}
where $H$ is the elasticity constant. A more realistic  model is the finite extensible nonlinear elasticity (FENE) model with
\begin{equation}\label{potential}
\Psi(m)= -\frac{Hb}{2}\log \left(1-\frac{|m|^2}{b}\right), \quad m\in B.
\end{equation}
Here $B\df B(0,\sqrt b)$ is the ball with center $0$ and radius
$\sqrt{b}$ which denotes the maximum dumbbell extension. In this work we
shall focus our attention on the potential \eqref{potential} and the
case $b>2$, which is known to contain the parameter range of
physical interest. We refer the reader to \cite{DE:1986,BCAH87} for
a comprehensive survey of the physical background.

In past years the well-posedness of the FENE model
\eqref{n1}-\eqref{n3} has been studied intensively in several
aspects. For local well-posedness of strong solutions we refer the
reader to \cite{JLL04} for the FENE model (in the setting where the
Fokker-Planck equation is formulated by a stochastic differential
equation) with $b > 2$ or sometime $b > 6$, \cite{ETiZh04} for a
polynomial force and  \cite{ZhZh06} for the FENE model with $b >
76$. For a preliminary study on some related coupled PDE systems, we
refer to the earlier work \cite{Rm91} (however, the FENE model was
not addressed there). Moreover, the authors in \cite{LLZ07} proved
global existence of smooth solutions near equilibrium under some
restrictions on the potential; further developments were made
in subsequent works \cite{LiZh08, LiZhZh08}.  
More recently, N. Masmoudi
\cite{Ma08} proved global existence for the FENE model
\eqref{n1}-\eqref{n3} for a class of potentials (\ref{potential})
with $b>0$ assuming  that the data is small, or the model is
restricted to the co-rotational case in dimension two.

For results concerning the existence of weak solutions to coupled
Navier-Stokes-Fokker-Planck systems and a detailed survey of related
literature we refer to \cite{BaSchSu05,LiMa07,LZ07,BaSu07,BaSu08}.
For an earlier result on existence of weak solutions, we refer to
\cite{DqLcYp05} for the Fokker-Planck equation alone with $b>4$. On
the other hand, the authors in \cite{JLLO06}, investigated the
long-time behavior of both Hookean models and FENE models in several
special flows in a bounded domain with suitable boundary conditions.



The complexity with the FENE potential lies mainly with the
singularity of the equation at the boundary. To overcome this
difficulty, several transformations relating to the equilibrium solution have been
introduced in literature. See, e.g. \cite{CL04,CL04+, DqLcYp05, LiLi08, KS09}. A detailed discussion will be given in Section 2.
In \cite{LiLi08}, C. Liu and H. Liu closely examined the necessity
of Dirichlet boundary conditions for the microscopic FENE model. By
the method of the Fichera function they were
able to conclude that $b=2$ is a threshold in the sense that for
$b>2$ any preassigned boundary value of the ratio of the
distribution and the equilibrium will become redundant, and for
$b<2$ that value has to be a priori given. For the microscopic FENE
model, singularity in the potential requires at least the zero
Dirichlet boundary condition
\begin{equation}\label{zero}
  f|_{\p B}=0.
\end{equation}
This is consistent with the result in \cite{JoLe03}, which states that the stochastic solution trajectory does not reach the boundary almost surely.

The boundary issue for the underlying FENE model is fundamental, and our main quest in this paper is
whether one can identify a sharp boundary requirement so that both existence and uniqueness of a global weak solution
to the microscopic FENE model can be established, also the solution remains a probability density. The answer is positive, and we claim that $f$ must satisfy the following
boundary condition
\begin{equation}\label{n6}
fd^{-1}\left|_{\p B}\right. =0 \quad \text{for~almost~all~} t>0,
\end{equation}
where $d\df d(m,\p B)$ denotes the distance function from $m \in B$
to the boundary $\p B$. Our claim is supported by our main results:
the global well-posedness for the Fokker-Planck equation stated in
Theorem \ref{thm2},  the property of the solution as a probability density given in Proposition \ref{prop+}, and the sharpness of (\ref{n6}) stated in
Proposition \ref{prop2}.

In this article, we focus on the underlying Fokker-Planck equation
\eqref{n3} alone. Let $v(t,x)$ be the velocity field governed by \eqref{n1} and \eqref{n2}. We assume that this underlying velocity field is smooth, then a simplification can be made by considering the
microscopic model (\ref{n3}) along a particle path defined as
\begin{equation*}
\p_t \mathrm{X}(t,x)=v(t,\mathrm{X}(t,x)),\quad \mathrm{X}(t=0,x)=x.
\end{equation*}
For each fixed $x$, the distribution function $\tilde f
(t,m;x)\df f(t,\mathrm{X}(t,x),m)$ solves
\begin{equation*}
\p_t\tilde f+ \nabla_m\cdot(\nabla v m \tilde f)=\frac2\zeta
\nabla_m \cdot(\nabla_m \Psi(m)\tilde f)+\frac{2kT}{\zeta}\Delta_m
\tilde f.
\end{equation*}
By a suitable scaling (\cite{LiLi08}), and denote $\tilde f$ still by $f(t, m)=\tilde f(t, m; x)$,
we arrive at the following equation
\begin{equation}\label{n4}
\p_t f+ \nabla\cdot(\kappa m f)=\frac12 \nabla
\cdot\left(\frac{bm}{\rho} f \right)+\frac12\Delta f.
\end{equation}
Here, $\rho=b-|m|^2$ and $\kappa(t)= \nabla v(t, \mathrm{X}(t,x))$
is a bounded  matrix such that $\text{Tr}(\kappa)=0$. We omit $m$
from  $\nabla_m$  in (\ref{n4}) for notational convenience. In this
paper we prove well-posedness of (\ref{n4}) subject to some side
conditions. The well-posedness of the full coupled system
\eqref{n1}-\eqref{n3} is the subject of a forthcoming paper
\cite{LiSh09}.

A weak solution of the Fokker-Planck equation
\eqref{n4} with the initial condition
\begin{equation}\label{initial}
f(0,m)=f_0(m), \quad m \in B,
\end{equation}
and boundary requirement \eqref{n6} is defined in the following.
\begin{defn}\label{def1} We say $f$ is a weak solution of \eqref{n4}, \eqref{initial}, and \eqref{n6} if the following conditions are
satisfied:

For an arbitrary subdomain $B'$ of $B$ such that
$\ol{B'}\subset B$ and almost all $t \in (0,T)$,
\begin{enumerate}
\item[(1)] $f \in L^2(0,T;H^1(B')) ~and~ \p_t f \in
L^2(0,T;H^{-1}(B'))$,
\item[(2)] for any $\varphi\in C^1_c(B)$,
\begin{equation}\label{2.2}
\int_{B} \left[ \p_tf \varphi  - f \kappa m \cdot \nabla \varphi +
\frac{b f m \cdot \nabla \varphi}{2\rho} + \frac12 \nabla f
\cdot\nabla \varphi \right] dm =0,
\end{equation}
\item[(3)]
\begin{equation}\label{2-3}
f(0,m) = f_0(m)\quad \text{in~} L^2(B'),
\end{equation}
\item[(4)] and for $B_r\df B(0,r)$,
\begin{equation}\label{bdy3}
\lim_{r\to\sqrt b} ||fd^{-1}\left|_{\p B_r}\right.||_{L^2(\p B_r)} =
0.
\end{equation}
\end{enumerate}
\end{defn}
\noindent Note that \eqref{2-3} makes sense since $f \in C([0,T];
L^2(B'))$ implied by (1) above, and also $fd^{-1}\left|_{\p
B_r}\right.$ is well defined in $L^2(\p B_r)$ by the standard trace
theorem.

Regarding the weak solution defined above, several remarks are in order.
\begin{itemize}
\item  The reason for taking compactly supported functions as test functions in Definition 1
is that we want to avoid any priori restriction to a particular
weighted Sobolev space. It is this treatment that allows us to prove
sharpness of boundary condition (\ref{bdy3}).

\item  Boundary condition
\eqref{bdy3} ensures that $f(t,\cdot) \in L^1(B)$ for each $t$.
Indeed, we can choose $r_0 \in (0,\sqrt b)$ such that if $r\geq r_0$
\begin{equation*}
||fd^{-1}|_{\p B_r}||_{L^2(\p B_r)} \leq 1.
\end{equation*}
Then
\begin{eqnarray*}
\int_B |f| dm &=& \int_{B_{r_0}} |f| dm +  \int_{r_0}^{\sqrt b}
\int_{\p B_{r}} |f| dS dr\\
& \leq & C_1 ||f||_{L^2(B_{r_0})} + C_2 \int_{r_0}^{\sqrt b}
||fd^{-1}|_{\p B_r}||_{L^2(\p B_r)} dr <\infty.
\end{eqnarray*}
\item  Boundary condition (\ref{bdy3}) or  \eqref{n6} is, in its
type, different from the zero flux boundary condition
\begin{equation}\label{bdy}
\left(\frac{bmf}{2\rho}  + \frac{1}{2}\nabla_m f - \kappa m
f\right)\cdot \frac{m}{|m|} \Big|_{\p B}=0,
\end{equation}
which is known to preserve the conservation property, and has been used in many priori works. The relation of these two types
of boundary conditions will be discussed in Section 2 as well.
\end{itemize}
 In order to establish an existence
theorem, we now identify a subspace of $H^1(B)$ with an appropriate
weight to incorporate boundary requirement (\ref{bdy3}). For
simplicity, we consider the case with trivial velocity field such
that $\kappa=0$, then equation (\ref{n4}) becomes
\begin{equation*}
\p_t f= \frac12\nabla\cdot \left(\rho^{b/2} \nabla
\left(\frac{f}{\rho^{b/2}}\right)\right).
\end{equation*}
It follows from this conservative form that the only equilibrium
solution $f^{eq}$ must be a multiplier of $\rho^{b/2}$, i.e.
$$
f^{eq}=Z^{-1}\rho^{b/2},
$$
where $Z$ is a normalization factor such that $\int_B f^{eq} dm=1$.

We are interested in the case
\begin{equation}\label{condition b}
b>2.
\end{equation}
In such a  case $ f^{eq}$ satisfies boundary requirement (\ref{n6}).
Moreover
\begin{equation*}
f^{eq} \in H^1_{-b/2}(B).
\end{equation*}
Here, $H^1_{-b/2}(B)=\{\phi: \phi, \p_{m_j}\phi \in L^2_{-b/2}(B) \}$ with
\begin{equation*}
L^2_{-b/2}(B) =\left\{\phi:\int_B \phi^2 (b-|m|^2)^{-b/2}dm <\infty \right\}.
\end{equation*}

Our main results are summarized in Theorem \ref{thm2}, Proposition \ref{prop+}  and
Proposition \ref{prop2} below.
\begin{thm}\label{thm2}
Assume \eqref{condition b} and $\kappa(t)\in C[0, T]$ for a given $T>0$.  \\
\noindent (i)  If
\begin{equation}\label{ini}
f_0(m) \in L^2_{-b/2} (B) ,
\end{equation}
then there exists a unique solution $f$ of \eqref{n4},
\eqref{initial}, and \eqref{n6} in the sense of
Definition \ref{def1}.  Moreover,
\begin{equation}\label{cont}
\dis\max_{0\leq t\leq T} ||f(t,\cdot)||_{L^2_{-b/2}(B)} + ||f||_{L^2(0,T;H^1_{-b/2}(B))} + ||\p_t f||_{L^2(0,T;(({H^1}_{-b/2})^*(B))} \leq C||f_0||_{L^2_{-b/2}(B)}.
\end{equation}
\noindent (ii) For any
\begin{equation*}
f_0 (m) \in L^2_{loc}(B),
\end{equation*}
there exists at most one solution $f$ .
\end{thm}
\begin{proof}The proof of $(i)$ will be done in  Section 3 - 5. In order to prove $(ii)$ we assume that $f_1, f_2$ are two weak solutions
of the problem with arbitrary initial data $f_0(m)$. Then $f_1-f_2$
solves \eqref{n4} with zero initial data which is in $L^2_{-b/2}$.
From \eqref{cont} in $(i)$ it follows  that $f_1 \equiv f_2$ in
$L^2(0,T;H^1_{-b/2})$.
\end{proof}
\noindent We remark that the restriction on $b$ in \eqref{condition b}
is essential to obtain the energy estimate \eqref{cont}.

The weak solution thus obtained is indeed a probability density. More precesely we have the following.
\begin{prop}\label{prop+}  Let $f$ be a weak solution to \eqref{n4}, \eqref{initial}, and \eqref{n6} defined in Definition 1 subject to condition
\eqref{ini}.  Then,
\begin{equation}\label{con}
\int_B f(t, m)dm=\int_B f_0(m) dm, \quad \forall t>0.
\end{equation}
Furthermore if $f_0(m) \geq 0$ a.e. on $B$,   then $f(t,\cdot) \geq 0$ a.e. on $B$ for all $t>0$.
\end{prop}
This proposition  will be proved in Section 2.

The following proposition states that boundary condition \eqref{n6}
is sharp for the uniqueness of the weak solution.
\begin{prop}\label{prop2}Assume \eqref{condition b} and $\kappa(t)\in C[0, T]$.
If boundary condition \eqref{n6} fails, that is,
\begin{equation*}
fd^{-1}\left|_{\p B}\right. \neq 0
\end{equation*}
is assumed, then the Fokker-Planck equation \eqref{n4} with
$f_0(m)=0$ has infinitely many solutions.
\end{prop}

In other words, Proposition \ref{prop2} implies that part $(ii)$ in
Theorem \ref{thm2} would fail if boundary requirement \eqref{n6}
were weaken so  that near boundary the distribution approaches zero
not faster than the distance function.

The justification of sharpness follows from the existence of a
 Cauchy-Dirichlet problem for $w$ defined by
\begin{equation}\label{transformation}
w=\frac{f}{\rho}-g
\end{equation}
with $g$ being a class of functions properly constructed.

 This article is organized as follows.
In Section 2, we prove Proposition \ref{prop+} and provide some
preliminaries including: (1) several transformations used to handle
the boundary difficulty, (2) equivalence of two weighted function
spaces,   and (3) the relation of our  boundary condition to the
natural flux boundary condition. In Section 3, we transform the
Fokker-Planck equation to certain Cauchy-Dirichlet problem, named as
$W$-problem, and define a weak solution of $W$-problem in a weighted
Sobolev space. The well-posedness of the $W$-problem is shown in
Section 4 by the Galerkin method and the Banach fixed point theorem.
This leads to the well-posedness of the Fokker-Planck equation,
Theorem \ref{thm2}; details of the proof are presented in Section 5.
In Section 6, we construct non-trivial solutions for the
Fokker-Planck equation described in Proposition \ref{prop2}.

\section{Preliminaries}
\subsection{Probability density }
With the definition of our weak solution given in Definition 1 we shall
show that $f$ has the usual properties of a probability density
function (i.e. it is non-negative and has a unit integral over $B$
for all $t>0$ if it is so initially) -- this is to prove  Proposition \ref{prop+}.

Given $f_0$ in $L^2_{-b/2}(B)$ and $f_0 \geq 0$ a.e. on $B$, we
define $ f_{0, l}=\eta_l* f_0 \in C^\infty(B) $ for $l \geq 1$. Here
$\eta_l(m)=l^d\eta(lm)$ denotes the usual scaled mollifier.   We have
$$
\lim_{l \to \infty}  \|f_{0, l}-f_0\|_{L^2_{-b/2}(B)} =0, \quad f_{0, l}\geq 0.
$$
Suppose that $f_l$ is the weak solution of \eqref{n4},
\eqref{initial}, and \eqref{n6} subject to initial condition
$f_l(0,m)=f_{0,l}(m)\in C^\infty(B)$. Then, for any $T>0$  and $0<t<T$,
\begin{eqnarray}
\left|\int_B f(t,m) -f_l(t,m) dm \right| &\leq& C \max_{0\leq t\leq
T}
\|f(t,\cdot)-f_l(t,\cdot)\|_{L^2_{-b/2}(B)} \\ \notag
&\leq& C \|f_0-f_{0,l}\|_{L^2_{-b/2}(B)}. \label{ff}
\end{eqnarray}
Hence for justification of the conservation of polymers,  it suffices  to prove that
\begin{equation}\label{nco}
\int_B f_l(t,m) dm =\int_{B} f_l(0, m)dm, \quad \forall t\geq 0.
\end{equation}
To do so, we take a test function $\varphi_\varepsilon \in C^\infty_c(B)$
converging to $\chi_B$ as $\varepsilon \to0$ such that
\begin{equation*}
\varphi_\varepsilon(m)=\left\{
       \begin{array}{ll}
         1, & |m| \leq \sqrt b- \varepsilon \\
         0, & |m| \geq \sqrt b-  \varepsilon/2
       \end{array}
     \right.
\end{equation*}
and
\begin{equation}\label{cutoff}
|\p_{m_i} \varphi_\varepsilon| < C \frac{1}{\varepsilon},\quad
|\p_{m_i}\p_{m_j} \varphi_\varepsilon| < C \frac{1}{\varepsilon^2}.
\end{equation}
From (\ref{2.2})
and the fact that derivatives of $\varphi_\varepsilon$ are supported
in $B^\varepsilon: = B_{\sqrt b-\varepsilon/2}\setminus B_{\sqrt
b-\varepsilon}$, we have
\begin{equation}\label{dtf}
\int_B \p_t f_l \varphi_\varepsilon dm = \int _{B^\varepsilon}
\left[f_l \kappa m \cdot \nabla \varphi_\varepsilon - \frac{b f_l m
\cdot \nabla \varphi_\varepsilon}{2\rho} - \frac12 \nabla f_l
\cdot\nabla \varphi_\varepsilon \right] dm.
\end{equation}
 Applying the mean value theorem of the form
$$
\int_{B^\varepsilon }g dm = \frac{\varepsilon}{2}\int_{\partial
B_r}g dS, \quad \text{for~some~}r \in (\sqrt{b}-\varepsilon,
\sqrt{b}-\varepsilon/2),
$$
to the first term on the right of (\ref{dtf}) together with
\eqref{cutoff}, we obtain
$$
\left|\frac{\varepsilon}{2} \int_{\p B_{r}} f_l\kappa m\cdot \nabla
\varphi_\varepsilon dS\right| \leq C \int_{\p B_{r}} |f_l| dS \leq C
\|f_ld^{-1} \|_{L^2(\p B_r)}.
$$
Similarly the second term on the right of (\ref{dtf}) is bounded by
$$
\int_{\p B_{r}} \left|\frac{f_l}{\rho}\right|dS \leq C
\|f_ld^{-1}\|_{L^2(\p B_r)}.
$$
It follows from  \eqref{bdy3} that the above two
upper bounds converge to zero as $\varepsilon \to 0$ .

Integration by parts in the last term in (\ref{dtf}) yields
\begin{eqnarray*}
  \left|\int_{B^\varepsilon} \nabla f_l \cdot\nabla \varphi_\varepsilon d m \right| & \leq &
\int_{B^\varepsilon} |f_l \Delta \varphi_\varepsilon| d m + \int_{\p
B^\varepsilon} |f_l \nabla
\varphi_\varepsilon | dS \\
 &= & \frac{\varepsilon}{2}\int_{\partial B_r}  |f_l\Delta \varphi_\varepsilon| d S + \int_{\p
B^\varepsilon} |f_l \nabla
\varphi_\varepsilon | dS \\
 &\leq & C \int_{\p B_r \cup \p B^\varepsilon} \frac{|f_l|}{\varepsilon}dS,
\end{eqnarray*}
which, in virtue of $ |f_l|/\varepsilon \leq |f_l|d^{-1} $ on $\p
B_r \cup \p B^\varepsilon$, is converging to zero as $\varepsilon
\to 0$ as well.

Due to Theorem \ref{thm2} and the initial condition $f_{0,l} \in
C^\infty(B)$, it follows that $\p_t f_l$ is bounded in any $B_{r}$
for $0<r<\sqrt b$. Thus, for any $\tau,s >0$
\begin{eqnarray*}
\left|\int_B f_l (\tau,m) \varphi_\varepsilon dm - \int_B f_l(s,m)
\varphi_\varepsilon dm \right| &=& \left| \int_s^\tau \frac{d}{d t}
\left(\int_B
f_l(t,m) \varphi_\varepsilon dm \right) dt \right|\\
 &=& \left| \int_s^\tau
 \int_B \p_t f_l(t,m) \varphi_\varepsilon dm dt \right|.
\end{eqnarray*}
Using the estimate for $\int_B \p_t f_n(\tau,m) \varphi_\varepsilon
dm$ together with the boundedness of $f_l(t,m)$, we can send
$\varepsilon$ to zero to obtain (\ref{nco}) as  claimed.

We now turn to justify  the positivity.  Consider the transformation
introduced in \cite{LiLi08}
\begin{equation}\label{fw}
f_l=w_l\rho^{b/2-\alpha}e^{Kt}.
\end{equation}
Then $w_l$ solves
\begin{equation}
\rho^2\p_t w_l -\frac12 \rho^2 \Delta w_l - \rho
[\frac{4\alpha-b}2m-\rho \kappa m] \cdot \nabla w_l - c(m) w_l=0,
\end{equation}
where
$$
c(m)=- K\rho^2  + \alpha [nb + (2\alpha  + 2 - n - b)|m|^2 ] +
(b-2\alpha)\rho m \cdot \kappa m.
$$
Then for any $\overline{B_r} \subset B$, $w_l$ is a classical
solution in $(0,T]\times B_r$.
It was shown in \cite{LiLi08}  that there exist  $\alpha < b/2-1$ and $K$ so that
$c(m)<0$.
 The maximum principle yields  that
$w_l$ can not achieve  a negative minimum at the interior points of
$[0,T]\times B_r$. Thus the negative minimum of $w_l$, if it exists,
can only be attained on the parabolic boundary of the domain.

From the transformation \eqref{fw} and the condition $b/2-\alpha >
1$, it follows that the negative minimum of $f_l$, if any, can only  be attained at the initial time. Therefore
\begin{equation*}
 f_l \geq -\max  f_{0,l}^- \geq 0.
\end{equation*}

Now fix $t$. For any $x_0$ and $\eta$ such that $B(x_0;\eta) \subset
B$,
\begin{equation*}
-\int_{B(x_0;\eta)} f dm + \int_{B(x_0;\eta)} f_l dm \leq
\int_{B(x_0;\eta)} |f-f_l| dm \leq \int_B |f-f_l| dm \leq  C
||f_0-f_{0,l}||_{L^2_{-b/2}}.
\end{equation*}
Here (\ref{ff}) has been used to obtain the last inequality.  Hence
\begin{equation*}
\int_{B(x_0;\eta)} f dm \geq \int_{B(x_0;\eta)} f_l dm
-C||f_{0,l}-f_0||_{L^2_{-b/2}},
\end{equation*}
which as $l\to\infty$ leads to
\begin{equation*}
\int_{B(x_0,\eta)} f\geq 0.
\end{equation*}
Since $x_0$ and $\eta$ are arbitrary, $f\geq 0$ almost everywhere on $B$ for $t>0$.  The proof of Proposition \ref{prop+} is now complete.


\subsection{Transformations }
To overcome the difficulty caused by the boundary singularity,
several transformations  have been introduced in literature. With
boundary condition (\ref{n6}),  in this work we introduce
$$
w=\frac{f}{\rho}
$$
to transform the Fokker-Planck equation to a degenerate parabolic equation with zero boundary
condition (see details in Section 3).  A widely accepted transformation is the ratio of the unknown to the equilibrium solution, i.e.,
$$
w=\frac{f}{\rho^{b/2}}.
$$
Such a transformation was used in \cite{LiLi08} to reformulate the Fokker-Planck equation, and examine whether
a Dirichlet type boundary condition is necessary.

A third transformation  is
$$
w=\frac{f}{\rho^{b/4}}.
$$
This was used in  \cite{DqLcYp05, Guo03} to remove the singularity
at the boundary in the resulting equation. It was
also used in \cite{KS09} to formulate a weak formulation of $w$ for
discretization using a spectral Galerkin approximation.

Another  transformation defined by
$$
w=\frac{f}{\rho^s}
$$
with $b\geq 4s^2/(2s-1)$ and $s>1/2$ is said to also lead to a well-posed
problem. The minimum value of the function $4s^2/(2s-1)$ is attained
at $s=1$, yielding the maximum range of $b$ values, $b\geq 4$. This
transformation was proposed in \cite{CL04, CL04+} in the special
case $s=2$ and $s=2.5$, where these values were chosen on the basis
of numerical experiments in two and three dimensions, respectively.
We note that our transformation $w=f/\rho$ corresponds to $s=1$, but
not limited by $b\geq 4$.

\subsection{Weighted Sobolev spaces} In contrast to the standard weighted Sobolev space $H^1_{-b/2}(B)$ used in this work,
the following weighted function space
\begin{equation*}
\rho^{b/2} H^1_{b/2}(B):= \left\{\phi: \quad \frac{\phi}{\rho^{b/2}}
\in H^1_{b/2}(B)
\right \}
\end{equation*}
is well known in literature for Fokker-Planck equations with FENE
potentials, see e.g. \cite{BaSchSu05, JLLO06, BaSu07, LiMa07, Ma08,
KS09}. We now show their equivalence as long as $b >2$.

The key estimate we need to prove the equivalence is the embedding
theorem stated in Lemma \ref{lem1}.
Set  $\psi=\phi\rho^{-b/2}$. If $\phi \in H^1_{-b/2}(B)$, we use the
relation
$$
\nabla \psi= \frac{\nabla \phi}{\rho^{b/2}} -\frac{2m}{\rho^{b/2 +1}}\phi.
$$
It is obvious that
$$
\frac{\nabla \phi}{\rho^{b/2}} \in L^2_{b/2}(B).
$$
Also  the use of  Lemma \ref{lem1} and the fact that
$H^1_{-b/2}(B)= \oc H{}^1_{-b/2}(B)$ for $b>2$ (see \cite{Ku85})
give
$$
\left \| \frac{\phi}{\rho^{b/2 +1}}  \right \|_{L^2_{b/2}} =\|\phi\|_{L^2_{-b/2-2}}\leq C \|\phi\|_{H^1_{-b/2}}.
$$
Hence  $\phi \in  \rho^{b/2} H^1_{b/2}(B)$.  If $\phi \in \rho^{b/2}
H^1_{b/2}(B)$ we use the following identity
$$
\nabla \phi=\rho^{b/2}\nabla \psi +2m \rho^{b/2-1}\psi.
$$
It is easy to see that $\rho^{b/2}\nabla\psi \in L^2_{-b/2}$; also
for $b>2$  we have
$$
\|\rho^{b/2-1}\psi\|_{L^2_{-b/2}}=\|\psi\|_{L^2_{b/2-2}}\leq C \|\psi\|_{H^1_{b/2}}
$$
by Lemma \ref{lem1}.  Thus $\phi \in H^1_{-b/2}(B)$. These together
verify that $\rho^{b/2} H^1_{b/2}$ and  $H^1_{-b/2}$ are equivalent
when $ b >2$.

\subsection{Boundary conditions}
Granted certain smoothness of $f$, e.g. $f\in C^1(\bar B)$,  one may
argue that our boundary condition (\ref{n6})  is equivalent to the
zero flux boundary condition (\ref{bdy}).

Set  $\nu=\frac{m}{|m|}$ and $g=fd^{-1}$.  We calculate the flux
\begin{align*}
J& :=\left( \frac{bmf}{2\rho}-\kappa m f +\frac{1}{2}\nabla
f\right)\cdot \frac{m}{|m|}\\
&=\frac{b|m|}{2(|m|+\sqrt{b})}\frac{f}{d}+\frac{1}{2} \frac{\partial
f}{\partial \nu}-|m|\nu\cdot \kappa \nu f.
\end{align*}
Due to singularity on boundary it is necessary that $f|_{\partial
B}=0$. For any point $p\in
\partial B$, let $m$ be a point in $B$ such that $m+d\nu=p$. Then
$$
\frac{\partial f}{\partial \nu}(p)=-\lim_{d \to 0} \frac{f(m)}{d}.
$$
We thus have
$$
J(p)= \lim_{d\to 0}J(m)=  \frac{1}{4}(b-2)\lim_{d \to 0} \frac{f(m)}{d}.
$$
For $b\not=2$, this implies that $J(p)=0$ if and only if
$$
\lim_{d \to 0} \frac{f(m)}{d}=0.
$$

\section{Transformation of the microscopic FENE model}
In what follows we shall call the Fokker-Planck equation (\ref{n4}) with initial condition (\ref{initial}) and boundary condition (\ref{n6}) as the Fokker-Planck-FENE (FPF) problem. We first formulate a time evolution equation from the FPF problem. Define $w(t,m)$  as
\begin{equation}\label{2-6}
f(t,m)=w(t,m)\rho.
\end{equation}
Then \eqref{n4} is transformed to
\begin{equation}\label{2-7}
\rho\left[\p_tw -\frac12 \Delta w -\frac{(b-4)m-2\rho \kappa m
}{2\rho} \cdot \nabla w -\frac{c}{\rho} w\right]=0,
\end{equation}
where
\begin{equation}\label{c(m)}
 c(t,m)= 2 m \cdot \kappa(t) m +   n (b/2-1).
\end{equation}
Setting a parameter
\begin{equation*}
\beta=-\frac{b}{2}+2,
\end{equation*}
we rewrite \eqref{2-7}  as
\begin{equation*}\label{2-8}
\rho^{b/2-1} \left[  \p_tw\rho^\beta -\frac12\nabla\cdot( \nabla
w\rho^\beta) + \kappa m\cdot \nabla w \rho^\beta -  c w
\rho^{\beta-1}\right]=0.
\end{equation*}

\noindent The boundary condition \eqref{n6} implies that
$w(t,\cdot)$ satisfies a homogeneous boundary condition for almost
all $t$ since the distance function $d$ and $\rho$ are equivalent
(see \eqref{2.31}).

The FPF problem is formally transformed to the following $W$-problem:
\begin{eqnarray}
&\dis \p_tw\rho^\beta-\frac12\nabla\cdot(\nabla w\rho^\beta ) +
\kappa m\cdot \nabla w \rho^\beta - c w \rho^{\beta-1}=0,& \quad
(t,m)\in (0,T] \times
B,\label{2.11}\\
&\dis w(0,m)=w_0(m),& \quad m \in B,\label{2.12}\\
&\dis w(t,m)=0,&\quad (t,m)\in [0,T]\times \p B.\label{2.13}
\end{eqnarray}
Here,
\begin{equation*}
w_0(m)= f_0(m)\rho^{-1}
\end{equation*}
according to the transformation \eqref{2-6}.

In order to define a weak solution of $W$-problem we introduce a weighted Sobolev space $H^1(\Omega; \sigma)$ for a nonnegative measurable function $\sigma$ as a set of measurable function $\phi$ such that
\begin{equation*}
||\phi||^2_{H^1(\Omega;\sigma)}:= \int_{\Omega} (|\nabla \phi|^2 + \phi^2)\sigma dm <\infty.
\end{equation*}
Similarly, a weighted $L^2(\Omega; \sigma)$ can be defined. $\overset\circ H{}^1(\Omega; \sigma)$ denotes a completion of $C^\infty_c (\Omega)$ with $||\cdot||_{H^1(\Omega; \sigma)}$. It is obvious that
$H^1(\Omega;\sigma)$ and $\oc H{}^1(\Omega;\sigma)$ are Hilbert spaces with the inner product $\langle\cdot,\cdot\rangle_{H^1(\Omega; \sigma)}$ defined as
\begin{equation*}
\langle\phi_1,\phi_2\rangle_{H^1(\Omega; \sigma)}=\int_\Omega (\nabla \phi_1 \cdot \nabla \phi_2  + \phi_1\phi_2)\sigma dm
\end{equation*}
and
\begin{equation*}
\oc H{}^1(\Omega; \sigma)\subset H^1(\Omega;\sigma).
\end{equation*}

\noindent For notational convenient, we use $H^1_\mu(\Omega), \oc
H{}^1_{\mu}(\Omega)$ and $L^2_\mu(\Omega)$ for $H^1(\Omega;
\rho^\mu), \oc H{}^1(\Omega; \rho^\mu)$ and $L^2(\Omega; \rho^\mu)$
respectively. We also omit the domain $\Omega$ if it is obvious.

\begin{lem}\label{lem1}
Suppose that $\Omega=B$.
\begin{enumerate}
  \item[(1)] If $\phi \in \oc H{}^1_\mu$ for $\mu <1$, then
    \begin{equation}\label{2.27}
      ||\phi||_{L^2_{\mu-2}} \leq C_0 ||\phi||_{H^1_\mu}.
    \end{equation}
      If $\mu>1$, we have the same inequality for $\phi \in H{}^1_\mu$
  \item[(2)] If $\phi \in H^1_{\mu}$ for $\mu<1$, then the trace map
    \begin{eqnarray*}
       \T: &H^1_{\mu}(\Omega)& \to L^2(\p \Omega)\\
          &\phi & \mapsto \phi|_{\p \Omega}
    \end{eqnarray*}
    is well defined, i.e. it is a bounded linear map.\\
    In particular, for $\phi \in \oc H{}^1_{\mu}$
   \begin{equation}\label{2.23}
    \T (\phi)=0.
   \end{equation}
   \end{enumerate}
\end{lem}

\begin{proof}
In \cite{Ne62}(see also \cite{Ku85}), it was proved that
\begin{equation*}
\oc H{}^1(\Omega; d^\mu)\hookrightarrow L^2(\Omega; d^{\mu-2})
\end{equation*}
provided $\p\Omega$ is Lipschitz continuous. Recall that $d$ denotes the distance from $m$ to the boundary of $\Omega$.
\eqref{2.27} follows from
\begin{equation}\label{2.31}
\sqrt b d \leq \rho \leq 2 \sqrt b d.
\end{equation}

\noindent It is also known that the trace map $\mathcal{T}$ is well
defined for $0\leq \mu<1$ (\cite{Ne62,Ku85}). For $\mu<0$
\begin{equation*}
||\phi||_{H^1} \leq b^{-\mu/2} ||\phi||_{H^1_\mu}
\end{equation*}
since $\rho^\mu \geq b^\mu $ for all $m\in B$.
Therefore, $\T$ is well defined for $\mu<1$. \eqref{2.23} is obvious from the definitions of the trace map and $\oc H{}^1_\mu$.
\end{proof}

We now define a weak solution to $W$-problem in a standard
manner. Multiplication by a test function $\varphi\in
C^1_c(B)$ to the equation \eqref{2.11} and integration over $B$
yield
\begin{equation*}\label{2.26}
\int_{B}\left[ \p_tw \varphi \rho^\beta  + \frac12  \nabla w
\cdot \nabla \varphi \rho^\beta  +   \kappa m\cdot \nabla w
\varphi \rho^\beta  - c w \varphi {\rho^{\beta-1}}
\right]dm=0.
\end{equation*}
 This equation is well defined assuming that $\p_tw(t,\cdot) \in (\Hb)^*$, the dual space of $\Hb$, and $w(t,\cdot), \varphi \in \Hb$ due to the boundedness of $c$ and Lemma \ref{lem1}. Moreover,
\begin{equation*}
\Hb \subset L^2_{\beta}
\end{equation*}
implies
\begin{equation*}
\Hb \subset L^2_{\beta} \subset (\Hb)^*.
\end{equation*}
Thus
\begin{equation*}
 w(t,x) \in C([0,T]; L^2_\beta).
\end{equation*}
Here we identify $L^2_\beta$ with its dual space.

Let $(\cdot,\cdot)_H$ denote the paring of a Hilbert space $H$ with its dual space $H^*$ and
\begin{equation}\label{ww}
\bf L[ w,\varphi;t]=\frac12 \int_{B} \nabla w(t,m) \cdot\nabla \varphi \rho^\beta dm+ \int_{B} \kappa m\cdot \nabla w(t,m) \varphi \rho^\beta dm.
\end{equation}
We now describe the weak solution we are looking for.

\begin{defn}\label{defn1}
A function $w(t,m)$ such that
\begin{equation*}
 w(t,m) \in L^2(0,T; \Hb), ~with~  \p_tw(t,m) \in L^2(0,T; (\Hb)^*)
\end{equation*}
is a weak solution of $W$-problem, \eqref{2.11}-\eqref{2.13}, provided
\begin{enumerate}
  \item[(1)] For each $\varphi\in \Hb$ and almost every $0\leq t\leq
  T$,
  \begin{equation*}\label{2.24}
  \dis ( \p_tw(t,\cdot), \varphi)_{\Hb} + \bf L[w,\varphi;t]= \int_{B} c  w(t,m) \varphi \rho^{\beta-1}dm,
  \end{equation*}
   \item[(2)] $\dis w(0,m)=w_0(m)$ in $L^2_\beta$ sense. i.e.
  \begin{equation*}\label{2.25}
   \int_{B} |w(0,m)-w_0(m)|^2 \rho^\beta dm =0.
  \end{equation*}
\end{enumerate}
\end{defn}

The following energy estimate for $\bf{L}[w,w;t]$ for fixed $t$ is crucial.
\begin{lem}\label{lem2}
There exist positive constants $C_1$ and $C_2$ depending only on $b$
and $||\kappa||_{L^\infty(0,T)}$ such that
\begin{equation*}
C_1 ||w(t,\cdot)||^2_{H_\beta^1} \leq \bf L[w,w;t] + C_2 ||w(t,\cdot)||^2_{L^2_\beta}.
\end{equation*}
\end{lem}
\begin{proof}  Let $\phi=w$ in (\ref{ww})  and apply the Schwarz inequality
we arrive at the above estimate as desired.
\end{proof}
The well-posedness  of the W-problem is  stated in the following
\begin{lem}\label{lem3}
$W$-problem, \eqref{2.11}-\eqref{2.13}, is uniquely solvable in weak sense for $w_0 \in L^2_\beta$. Furthermore,
\begin{equation*}\label{Wcont}
\max_{0\leq t\leq T} ||w(t,\cdot)||_{L^2_\beta} +
||w||_{L^2(0,T;\Hb)} + ||\p_tw||_{L^2(0,T;(\Hb)^*)} \leq
C||w_0||_{L^2_\beta}.
\end{equation*}
\end{lem}
A detailed proof will be presented in next section.

\section{Well-posedness for the transformed problem}
In this  section, we show the well-posedness of the weak solution to
$W$-problem. For this aim, we consider the following $U$-problem
containing a non-homogeneous term $h(t,m)\in L^2(0,T;
L^2_{2-\beta})$.
\begin{eqnarray}
&\dis \p_tu\rho^\beta-\frac12\nabla\cdot( \nabla u\rho^\beta) +  \kappa m\cdot \nabla u \rho^\beta - h=0,& \quad (t,m)\in (0,T] \times
B,\label{u1}\\
&\dis u(0,m)=u_0(m),& \quad m \in B,\label{u2}\\
&\dis u(t,m)=0,&\quad (t,m)\in [0,T]\times \p B.\label{u3}
\end{eqnarray}
The weak solution of $U$-problem is defined similarly.
\begin{defn}
We say a function $u$ such that
\begin{equation*}\label{2.55-2}
 u \in L^2(0,T; \Hb), ~with~  \p_tu \in L^2(0,T; (\Hb)^*)
\end{equation*}
is a weak solution of $U$-problem provided
\begin{enumerate}
  \item[(1)] for each $\varphi\in \Hb$ and almost every $0\leq t\leq T$
  \begin{equation*}\label{2.53}
  \dis (\p_tu(t,\cdot), \varphi)_{\Hb} + \mathbf{L}[u,\varphi; t]= \int_B h(t,m)\varphi dm,
  \end{equation*}
  \item[(2)] $\dis u(0,m)=u_0(m)$ in $L^2_\beta$.
\end{enumerate}
\end{defn}
\noindent We remark that  $\dis \int_B h(t,m)\varphi dm$ is finite for any $h(t,\cdot) \in L^2_{2-\beta}$ since
$\varphi \in L^2_{\beta-2}$ from \eqref{2.27}. Thus $\dis \int_B h(t,m)\varphi dm$ can be understood as the $L^2_0$ inner product although $h(t,\cdot)$ may not belong to $L^2_0$.

 The well-posedness for $U$-problem follows from the standard Galerkin method.
\begin{lem}\label{thm1}
For given $h\in L^2 (0,T; L^2_{2-\beta})$ and $u_0\in
L^2_\beta$, $U$-problem has a unique weak solution. Moreover,
\begin{equation}\label{2.60}
\max_{0\leq t\leq T} ||u(t,\cdot)||_{L^2_\beta} + ||u||_{L^2(0,T;\Hb)} + ||\p_tu||_{L^2(0,T;(\Hb)^*)} \leq C(||h||_{L^2 (0,T; L^2_{2-\beta})}+
||u_0||_{L^2_\beta}).
\end{equation}
\end{lem}

\begin{proof}
 We first construct an approximate solution in a finite-dimensional space. Let $\{\phi_i\}$ be a basis of $\Hb$ and $L^2_\beta$.
  The existence of such a basis can be verified from the fact that $\Hb$ is a dense subset of $L^2_\beta$.
     Consider an approximation $ u_l (t,m) = \dis \sum_{i=1}^l d^l_i(t) \phi_i $, where $d^l_i$ satisfies
      \begin{eqnarray}
       &( \p_t u_{l}(t,\cdot), \phi_j)_{\Hb} + \mathbf{L}[u_l,\phi_j; t]= \langle
  h(t,\cdot),\phi_j\rangle_{L^2_0},\quad 1\leq j\leq l, \label{2.55}\\
       &\dis\sum_{i=1}^l d^l_i(0) \phi_i \to u_0 ~\text{in~} L^2_\beta, ~ \text{as~} l\to \infty.\label{2.56}
      \end{eqnarray}
      Since \eqref{2.55} and \eqref{2.56} form a system of linear differential equations, $\{d^l_i\}$ is uniquely determined for each $l$. We rewrite \eqref{2.55} as
      \begin{equation}\label{2.55-1}
      \langle\p_t u_{l}(t,\cdot), \phi_j\rangle_{L^2_\beta} + \mathbf{L}[u_l,\phi_j; t]=\langle h(t,\cdot),\phi_j\rangle_{L^2_0},\quad 1\leq j\leq l.
      \end{equation}
Apply $d^l_j$ to \eqref{2.55-1} and sum for $1\leq j\leq l$, then for almost every $t$
     \begin{equation*}
       \langle\p_tu_{l}(t,\cdot), u_l(t,\cdot)\rangle_{L^2_\beta} + \mathbf{L}[ u_l,u_l; t]=\langle h(t,\cdot),u_l(t,\cdot)\rangle_{L^2_0}.
\end{equation*}
From Lemma \ref{lem2}, it follows that
\begin{equation}\label{2.69}
      \frac{d}{dt}||u_l(t,\cdot)||^2_{L^2_\beta} +  2 C_1|| u_l(t,\cdot)||^2_{H^1_\beta}  \leq 2 C_2||u_l(t,\cdot)||^2_{L^2_\beta}+2 \langle
  h(t,\cdot), u_l(t,\cdot)\rangle_{L^2_0}.
     \end{equation}
From \eqref{2.27}, for any $\delta>0$
  \begin{equation*}
  |\langle h(t,\cdot), u_l(t,\cdot)\rangle_{L^2_0}| \leq \frac1{2 \delta}||h(t,\cdot)||^2_{L^2_{2-\beta}}  + \frac\delta 2 C_0^2|| u_l(t,\cdot)||^2_{H^1_\beta}.
  \end{equation*}
  With $\delta = C_1/C_0^2$, \eqref{2.69} can be rewritten as
 \begin{equation}\label{2.74}
  \frac{d}{dt}||u_l(t,\cdot)||^2_{L^2_\beta} +  C_1|| u_l(t,\cdot)||^2_{H^1_\beta}  \leq 2C_2||u_l(t,\cdot)||^2_{L^2_\beta}+ C_0^2/C_1 ||h(t,\cdot)||^2_{L^2_{2-\beta}},
 \end{equation}
 or
 \begin{equation*}\label{2.75}
 \frac{d}{dt}||u_l(t,\cdot)||^2_{L^2_\beta}  \leq  2C_2||u_l(t,\cdot)||^2_{L^2_\beta}+ C_0^2/C_1
  ||h(t,\cdot)||^2_{L^2_{2-\beta}}.
 \end{equation*}
 Use Gronwall's inequality to obtain
\begin{equation*}\label{2.76}
\max_{0\leq t\leq T}||u_l(t,\cdot)||^2_{L^2_\beta}  \leq C \left (||u_0||^2_{L^2_\beta} + || h||^2_{L^2(0,T;L^2_{2-\beta})}\right),
\end{equation*}
where $C$ is an appropriate constant which depends on $\beta$, $b$, $T$ and $|\kappa|$. On the other hand, integration of \eqref{2.74} from $0$ to $T$ together with above inequality yields
\begin{equation}\label{2.77}
||u_l||^2_{L^2(0,T;\Hb)} \leq C \left( ||u_0||^2_{L^2_\beta}+ ||h||^2_{L^2(0,T;L^2_{2-\beta})}\right).
\end{equation}
A similar argument to that in \cite{Ev98} gives us the estimate for $||\p_t u_l||$ as
\begin{equation*}\label{2.79}
|| \p_t u_{l}||^2_{L^2(0,T; (\Hb)^*)} \leq C \left( ||u_0||^2_{L^2_\beta}+ ||h||^2_{L^2(0,T;L^2_{2-\beta})}\right).
\end{equation*}
Here we have used \eqref{2.55} with $\phi \in \Hb$ such that $||\phi||_{H^1_\beta}\leq 1$ and \eqref{2.77}. By passing to the limit as $l\to\infty$ and a standard argument (e.g. see \cite{Ev98}), we have well-posedness for $U$-problem.
\end{proof}

 Now, we introduce a linear map $\A$ to connect $W$ and $U$-problems as
\begin{eqnarray*}
\A : &L^2(0,\tau;L^2_{\beta})& \to L^2(0,\tau; L^2_{2-\beta})\\
     &w&  \mapsto c  w\rho^{\beta-1}.
\end{eqnarray*}
Since $c$ is bounded,
\begin{eqnarray*}
||  c w(t,\cdot)\rho^{\beta-1}||^2_{L^2_{2-\beta}} &\leq& ||
c||^2_{L^\infty}\int_B w^2(t,\cdot)
\rho^{2\beta-2}\rho^{2-\beta} dm,\\
&=& ||c||^2_{L^\infty} \int_B w^2(t,\cdot) \rho^{\beta} dm.
\end{eqnarray*}
Thus, $\A$ is well defined and
\begin{equation*}
||\A(w_1) -\A(w_2)||^2_{L^2(0,T;L^2_{2-\beta})} \leq ||
c||^2_{L^\infty} ||w_1-w_2||^2_{L^2(0,T;L^2_{\beta})}.
\end{equation*}

\noindent We define another map $\F$ such that
\begin{eqnarray*}
\F: & C([0,\tau]; L^2_{\beta})& \to C([0,\tau]; L^2_{\beta})\\
    & w& \mapsto  u.
\end{eqnarray*}
Here, $\F (w)$ is given by the weak solution of $U$-problem with
\begin{equation*}
 h = \A (w),
\end{equation*}
and the initial condition
\begin{equation*}\label{2.111}
u_0(m)=w(0,m).
\end{equation*}

\noindent The map $\F$ is well defined from Lemma \ref{thm1} and the
definition of $\A$. Now we show that $\F$ is a contraction mapping
for sufficiently small $\tau$. Let
\begin{equation*}
u_1=\F(w_1), \quad u_2= \F(w_2).
\end{equation*}
From the energy estimate \eqref{2.60},
\begin{eqnarray*}
||u_1- u_2||^2_{C([0,\tau]; L^2_{\beta})} &\leq & C ||\A(w_1)- \A(w_2)||^2_{L^2(0,\tau;L^2_{2-\beta})} \\
&= & C \int_0^\tau || \A(w_1)(t,\cdot)-\A(w_2)(t,\cdot)||^2_{L^2_{2-\beta}} dt\\
&\leq & C \int_0^\tau || w_1(t,\cdot)-  w_2(t,\cdot)||^2_{L^2_\beta} dt\\
&= & C\tau || w_1- w_2||^2_{C([0,\tau]; L^2_{\beta})}.
\end{eqnarray*}
Thus, $\F$ has a unique fixed point $w$ in $C([0,\tau];
L^2_{\beta})$ and $w$ solves $W$-problem in a weak sense in
$(0,\tau]\times B$, if $C\tau<1$. We are able to continue this procedure to obtain the global
well-posedness for the above constant $C$ is
independent of $\tau$.

For the fixed point $w$, \eqref{2.60} and the boundedness of $\A$ imply
that for $t\in [0, \tau']$
\begin{align*}
 \max_{0\leq t\leq \tau'} ||w(t,\cdot)||_{L^2_\beta} & + ||w||_{L^2(0,\tau';\Hb)} + ||\p_tw||_{L^2(0,\tau';(\Hb)^*)}\\
\quad \leq& C||\A(w)||_{L^2 (0,\tau'; L^2_{2-\beta})}+ C||w_0||_{L^2_\beta}\\
\qquad \leq& C\tau' \max_{0\leq t\leq \tau'} ||w(t,\cdot)||_{L^2_\beta}+C||w_0||_{L^2_\beta}.
\end{align*}
We select a small $\tau'<T$ such that $C\tau' <1$. Then
\begin{equation*}
\max_{0\leq t\leq \tau'} ||w(t,\cdot)||_{L^2_\beta} + ||w||_{L^2(0,\tau';\Hb)} + ||\p_tw||_{L^2(0,\tau';(\Hb)^*)} \leq  C||w_0||_{L^2_\beta}.
\end{equation*}
Thus,
\begin{equation*}
||w(\tau',\cdot)||_{L^2_\beta} \leq C||w_0||_{L^2_\beta}
\end{equation*}
and
\begin{eqnarray*}
\max_{\tau'\leq t\leq 2\tau'} ||w(t,\cdot)||_{L^2_\beta} + ||w||_{L^2(\tau',2\tau';\Hb)} + ||\p_tw||_{L^2(\tau',2\tau';(\Hb)^*)} &\leq&  C||w(\tau',\cdot)||_{L^2_\beta}\\
&\leq& C^2||w_0||_{L^2_\beta}.
\end{eqnarray*}
Continuing, after finitely many steps we obtain an energy estimation similar to \eqref{2.60}.  The proof of Lemma \ref{lem3}  is thus complete.

\section{Well-posedness for the FPF problem}
In Section 2, we transformed the FPF problem to W-problem formally,
but it is not difficult to show that they are equivalent. Indeed,
one can verify that boundary condition \eqref{n6} in the sense of
\eqref{bdy3} for the FPF problem is equivalent to the null boundary
condition for $W$-problem.

For any test function $\varphi \in
C^1_c$,  the weak solution formulation for $f$ can be transformed to the weak solution formulation
for $w$, with $\varphi \rho^{b/2-1}$ as the test function.  This is valid since  $\varphi \rho^{b/2-1} \in C^1_c$  is dense in $\Hb$.  Such a justification
can be reversed, hence  the FPF problem  and $W$-problem are equivalent.


Now we seek the function space in which the weak solution $f$ to the FPF problem  belongs. Recall that $\beta=-b/2+2$. For fixed $t\in [0,T]$, \eqref{2.27} implies
\begin{eqnarray}
\int_B |f|^2 \rho^{-b/2} dm &=& \int_B |w|^2 \rho^{\beta} dm,\label{4.1}\\
\int_B (|f|^2 + |\nabla f|^2) \rho^{-b/2} dm &\leq& C \int_B (|w|^2 +
|\nabla w|^2) \rho^{\beta}dm. \label{4.2}
\end{eqnarray}
Also, for $\phi \in H^1_{-b/2}$ we have
\begin{equation}\label{4.3}
|(\p_t f, \phi)_{H^1_{-b/2}}| = |(\p_t w,
\rho^{-1}\phi)_{H^1_\beta}|\leq C \|\p_t w\|_{(H^1_{\beta})^*}
\|\phi\|_{H^1_{-b/2}}.
\end{equation}
 The estimate of the weak solution,
\eqref{cont} follows from Lemma \ref{lem3} together with
\eqref{4.1}-\eqref{4.3}. This finishes the proof of (i) of Theorem \ref{thm2}.

\section{Non-uniqueness}
In this section we show that (\ref{n6}) is sharp in the sense that more
solutions can be constructed if a weaker condition is imposed ---  this is to prove Proposition \ref{prop2}.

It suffices to construct more than one solution to the Fokker-Planck
equation with $f_0(m)=0$ and the assumption
\begin{equation}\label{bdy2}
||fd^{-1}\left|_{\p B_r}\right.||_{L^2(\p B_r)} \not =0  \text{~ as~}
r\to \sqrt b ~\text{~for~} t\in I.
\end{equation}
Here $I$ is a nonzero measurable set.  The idea is to consider  a
class of functions  $g(t,m)\in W^{2,\infty}((0,T)\times B)$ such
that $g(0,m)=0$ and $g(t,m)|_{\p B}\neq 0$ for $t>0$ (e.g.
$g(t,m)=t|m|^2$) and show that for each $g$ the following problem
has a solution.
\begin{eqnarray}
&\dis \p_tf + \nabla\cdot(\kappa m f)=\frac12 \nabla \cdot \left(\frac{bm}{\rho} f\right) +\frac12 \Delta f,& \quad \text{in~} (0,T]\times B,\label{1}\\
&\dis f(0,m)=0,& \quad m \in B,\label{2}\\
&\dis f(t,m)\rho^{-1} =g(t, m),&\quad \text{in~} (0,T]\times \p B.\label{3}
\end{eqnarray}
Note that $\beta=-b/2+2<1$,  we can choose a parameter $\gamma$ such that
\begin{equation}\label{gamma}
\max\{\beta,-1\}<\gamma<1.
\end{equation}
\noindent To proceed, we define
\begin{equation*}
w= \frac{f}{\rho}- g.
\end{equation*}
The resulting equation when multiplied by $\rho^{1-\gamma}$ leads to the following
\begin{eqnarray}
&\p_tw\rho^\gamma -\frac12\nabla\cdot (\nabla w \rho^\gamma) + (\beta-\gamma)m\cdot \nabla w \rho^{\gamma-1}+ \kappa m\cdot \nabla  w \rho^\gamma - \tilde h_0=0,& \label{w11}\\
&w(0,m)=0, \quad~ m\in B,&\label{w12}\\
&w(t,m)=0, \quad~(t,m) \in [0,T]\times \p B,&\label{w13}
\end{eqnarray}
where
\begin{equation*}
\tilde h_0(t,m)= c w \rho^{\gamma-1} -\p_t g\rho^\gamma
+\frac12\nabla\cdot (\nabla g \rho^\gamma) - (\beta-\gamma)m\cdot
\nabla g \rho^{\gamma-1}- \kappa m\cdot \nabla g \rho^\gamma+
c g \rho^{\gamma-1}
\end{equation*}
with
\begin{equation*}
 c(t,m)=2 m \cdot \kappa(t) m +  n(b/2-1).
\end{equation*}
Let
\begin{eqnarray*}
\A_0 : &L^2(0,\tau;L^2_{\gamma})& \to L^2(0,\tau; L^2_{2-\gamma})\\
     &w&  \mapsto  \tilde h_0.
\end{eqnarray*}
This is well defined since $\gamma>-1$ from  \eqref{gamma} and the assumption that $g \in W^{2,\infty}((0,T)\times B)$. From the same argument as that in Section 4, it follows that \eqref{w11}-\eqref{w13} has a unique
solution $w$ such that
\begin{equation*}
w(t,m) \in L^2(0,T; \Hg), \quad  \p_tw(t,m) \in L^2(0,T; (\Hg)^*),
\end{equation*}
provided the corresponding $U$-problem
\begin{eqnarray}
&\p_tu\rho^\gamma -\frac12\nabla\cdot (\nabla u \rho^\gamma) + (\beta-\gamma)m\cdot \nabla u \rho^{\gamma-1}+ \kappa m\cdot \nabla  u \rho^\gamma - h_0=0,& \label{u11}\\
&u(0,m)=0, \quad~ m\in B,&\label{u12}\\
&u(t,m)=0, \quad~(t,m) \in [0,T]\times \p B,&\label{u13}
\end{eqnarray}
has a solution for any $h_0\in L^2(0,T;L^2_{2-\gamma})$. Note that
$\gamma<1$ is essential in order that the trace of $w$ at the boundary is defined. Equation \eqref{u11} is of the form of \eqref{u1} but with an additional term  $(\beta-\gamma)m\cdot \nabla u \rho^{\gamma-1}$. We thus define
\begin{equation*}
\bf L_0[ u,\varphi;t]\df\frac12 \int_{B} \nabla u
\cdot\nabla \varphi \rho^\gamma dm+ (\beta-\gamma)\int_B m\cdot
\nabla u\varphi \rho^{\gamma-1}dm+ \int_{B} \kappa m\cdot \nabla u
\varphi \rho^\gamma dm.
\end{equation*}
We may obtain the existence and uniqueness for \eqref{u11}-\eqref{u13} from the same argument of the well-posedness for $U$-problem \eqref{u1}-\eqref{u3}, if there is an energy estimate of $\bf L_0[u,u;t]$ which is similar to $\bf L [u,u;t]$ in Lemma \ref{lem2}. Indeed, for $u \in L^2(0,T; \Hg)$
\begin{equation*}
\frac12 \int_B |\nabla u|^2\rho^\gamma dm = \bf L_0[u,u;t] - \frac{\beta-\gamma}2 \left(\int_B m\cdot \nabla
u^2 \rho^{\gamma-1} dm \right)-\int_B \kappa m\cdot \nabla uu
\rho^{\gamma}dm.
\end{equation*}
We now claim that
\begin{equation}\label{rev1}
\int_B m\cdot \nabla u^2 \rho^{\gamma-1} dm \leq 0.
\end{equation}
Given this together with $\gamma >\beta$ from \eqref{gamma} we have
\begin{eqnarray*}
\frac12 \int_B |\nabla u(\cdot, t)|^2\rho^\gamma dm &\leq& \bf
L_0[u,u;t] -\int_B
\kappa m\cdot \nabla u u \rho^{\gamma}dm\\
&\leq& \bf L_0[u,u;t] + ||\kappa||_{L^\infty(0,T)} \sqrt b \left(
\frac{1}{2\delta } \int_B |\nabla u|^2 \rho^{\gamma} dm +
\frac\delta 2 \int_B |u|^2\rho^\gamma dm\right)
\end{eqnarray*}
for any $\delta>0$. By taking $\delta > ||\kappa||_{L^\infty(0,T)}
\sqrt b$, we obtain
\begin{equation*}
C'_1 ||u(t,\cdot)||^2_{H_\gamma^1} \leq \bf L_0[u,u;t] + C'_2
||u(t,\cdot)||^2_{L^2_\gamma}
\end{equation*}
for appropriate constants $C'_1$ and $C'_2$.

To verify the claim \eqref{rev1}, we define the trace operator $\mathcal{T}_0$ such that
\begin{eqnarray*}
       \T_0: &H^1_{\gamma}(B)& \to L^2(\p B)\\
          &u & \mapsto u\rho^{\frac{\gamma-1}2}|_{\p B}.
\end{eqnarray*}
Integration by parts on \eqref{rev1} yields
\begin{equation*}
\int_B m\cdot \nabla u^2 \rho^{\gamma-1} dm = -\int_B u^2
\left(n\rho^{\gamma-1} + 2(1-\gamma)|m|^2 \rho^{\gamma-2}\right)dm +
\sqrt b\int_{\p B}u^2\rho^{\gamma-1}dS,
\end{equation*}
or
\begin{eqnarray*}
\sqrt b\int_{\p B}u^2\rho^{\gamma-1}dS &=& 2\int_B m\cdot \nabla
u u \rho^{\gamma-1} dm + \int_B u^2 \left(n\rho^{\gamma-1} +
2(1-\gamma)|m|^2 \rho^{\gamma-2}\right)dm\\
&\leq & C ||u||^2_{H^1_\gamma}.
\end{eqnarray*}
Thus $\T_0$ is well defined, and for $u\in \Hg$, $\mathcal{T}_0 (u)=0$. Finally we obtain
\begin{equation*}
\int_B m\cdot \nabla u^2 \rho^{\gamma-1} dm = -\int_B u^2
\left(n\rho^{\gamma-1} + 2(1-\gamma)|m|^2 \rho^{\gamma-2}\right)dm \leq 0.
\end{equation*}
 This shows that there is a unique weak solution $u$ of
\eqref{u11}-\eqref{u13}, and thus $w\in L^2(0,T;\Hg)$ of
\eqref{w11}-\eqref{w13}.

Finally, $f= (w+g)\rho$ is a solution of \eqref{1}-\eqref{3}
satisfying \eqref{bdy2} for each $g$.  Hence the uniqueness of
\eqref{bdy2}-\eqref{2} fails as stated in Proposition \ref{prop2}.

\section{Conclusions }
In this paper, we have identified  a sharp Dirichlet-type boundary
requirement  to establish global existence of weak solutions to the
microscopic FENE model which is a component of bead-spring type
Navier-Stokes-Fokker-Planck models for dilute polymeric fluids. Such
a boundary requirement states that the distribution near boundary
approaches  zero faster than the distance function.  With this
condition, we have been able to show the uniqueness of  weak
solutions in the weighted Sobolev space $H^1_{-b/2}(B)$, which is
equivalent to the widely adopted weighted function space
$\rho^{b/2}H^1_{b/2}(B)$ for Fokker-Planck equations with the FENE
potential.   Moreover,  this condition ensures that the distribution
remains a probability density.  The sharpness of the boundary
condition was shown by construction of infinitely many solutions
when the boundary requirement fails.  In other words,  such a
condition provides a threshold on the boundary requirement:  subject
to this condition or any stronger ones incorporated through a
weighted function space,  the Fokker-Planck dynamics will select the
physically relevant solution, which is a probability density, see
e.g. \cite{BaSchSu05, JLLO06, BaSu07, LiMa07, Ma08, KS09},  and
converges to the equilibrium solution $Z^{-1}\rho^{b/2}$
\cite{JLLO06};  any weaker boundary requirement may lead to more
solutions, each depending on the rate  of  $f/d$ near boundary.  A
detailed elaboration of boundary conditions  for the coupled
Navier-Stokes-Fokker-Planck model will be the goal of our work
\cite{LiSh09}.

\bigskip
\section*{Acknowledgments} Shin thanks Professor Paul Sacks for stimulating discussions on Proposition \ref{prop2}.
We thank the referee for valuable suggestions and pointing out two relevant references \cite{KS09} and
\cite{BaSu08}.  Liu's research was partially supported by the
National Science Foundation under Kinetic FRG grant DMS07-57227 and grant DMS09-07963.


\begin{thebibliography}{10}

\bibitem{BaSchSu05}
J.~W. Barrett, C.~Schwab, and E.~S{\"u}li.
\newblock Existence of global weak solutions for some polymeric flow models.
\newblock {\em Math. Models Methods Appl. Sci.}, 15(6):939--983, 2005.

\bibitem{BaSu07}
J.~W. Barrett, and E.~S{\"u}li.
\newblock Existence of global weak solutions to kinetic models of dilute polymers.
\newblock {\em Multiscale Model. Simul.}, 6:506--546, 2007.

\bibitem{BaSu08}
J.~W. Barrett, and E.~S{\"u}li.
\newblock Existence of global weak solutions to dumbbell models for dilute polymers with microscopic
cut-off.
\newblock {\em Math. Mod. Meth. Appl. Sci. }, 18:935--971, 2008.

\bibitem{BCAH87}
R.~B. Bird, C.~Curtiss, R.~C. Armstrong, and O.~Hassager.
\newblock {\em Dynamics of Polymeric Liquids, Volume 2: Kinetic Theory}.
\newblock Wiley Interscience, New York, 1987.

\bibitem{CL04}
C. Chauvi\`{e}re, and A.~ Lozinski.
\newblock Simulation of complex viscoelastic flows using
Fokker-Planck equation: 3D FENE model.
\newblock {\em J. Non-Newtonian Fluid Mech. }, 122:201--214, 2004.

\bibitem{CL04+}
C. Chauvi\`{e}re, and A.~ Lozinski.
\newblock Simulation of dilute polymer solutions using a
Fokker-Planck equation.
\newblock {\em J. Comput. Fluids.}, 33:687--696, 2004.

\bibitem{DE:1986}
M.~Doi and S.~F. Edwards.
\newblock {\em The Theory of Polymer Dynamics}.
\newblock {Oxford University Press}, Oxford, 1986.

\bibitem{DqLcYp05}
Q.~Du, C.~Liu, and P.~Yu.
\newblock F{ENE} dumbbell model and its several linear and nonlinear closure
  approximations.
\newblock {\em Multiscale Model. Simul.}, 4(3):709--731 (electronic), 2005.

\bibitem{ETiZh04}
W.~E, T.~Li, and P.~Zhang.
\newblock Well-posedness for the dumbbell model of polymeric fluids.
\newblock {\em Comm. Math. Phys.}, 248(2):409--427, 2004.

\bibitem{Ev98}
L.~C. Evans.
\newblock {\em Partial Differential Equations}, volume~19 of {\em Graduate
  Studies in Mathematics}.
\newblock American Mathematical Society, Providence, RI, 1998.





\bibitem{Guo03}
Y. Guo.
\newblock The Vlasov-Maxwell-Boltzman system near Maxwellians.
\newblock {\em Invent. Math.}, 253:593--630, 2003.

\bibitem{JoLe03}
B.~Jourdain and T.~Leli{\`e}vre.
\newblock Mathematical analysis of a stochastic differential equation arising
  in the micro-macro modelling of polymeric fluids.
\newblock In {\em Probabilistic methods in fluids}, pages 205--223. World Sci.
  Publ., River Edge, NJ, 2003.

\bibitem{JLL04}
B.~Jourdain, T.~Leli{\`e}vre, and C.~Le~Bris.
\newblock Existence of solution for a micro-macro model of polymeric fluid: the
  {FENE} model.
\newblock {\em J. Funct. Anal.}, 209(1):162--193, 2004.

\bibitem{JLLO06}
B.~Jourdain, T.~Leli{\`e}vre, C.~Le~Bris, and F.~Otto.
\newblock Long-time asymptotics of amultiscale model for polymeric fluid flows.
\newblock {\em Arch. Ration. Mech. Anal.}, 181:97--148, 2006.

\bibitem{KS09}
D.~J. Knezevic, and E.~S{\"u}li.
\newblock Spectral Galerkin approximation of Fokker-Planck equations
with unbounded drift.
\newblock {\em ESAIM: M2AN}, 42(3):445--485, 2009.

\bibitem{Ku85}
A.~Kufner.
\newblock {\em Weighted {S}obolev Spaces}.
\newblock A Wiley-Interscience Publication. John Wiley \& Sons Inc., New York,
  1985.
\newblock Translated from the Czech.

\bibitem{LiZhZh08}
F.~Lin, P.~Zhang, and Z.~Zhang.
\newblock On the global existence of smooth solution to the 2-{D} {FENE}
  dumbbell model.
\newblock {\em Comm. Math. Phys.}, 277(2):531--553, 2008.

\bibitem{LLZ07}
F.-H. Lin, C.~Liu, and P.~Zhang.
\newblock On a micro-macro model for polymeric fluids near equilibrium.
\newblock {\em Comm. Pure Appl. Math.}, 60(6):838--866, 2007.

\bibitem{LiZh08}
F.~H. Lin and P.~Zhang.
\newblock The {FENE} dumbbell model near equilibrium.
\newblock {\em Acta Math. Sin. (Engl. Ser.)}, 24(4):529--538, 2008.

\bibitem{LiMa07}
P.-L. Lions and N.~Masmoudi.
\newblock Global existence of weak solutions to some micro-macro models.
\newblock {\em C. R. Math. Acad. Sci. Paris}, 345(1):15--20, 2007.

\bibitem{LiLi08}
C.~Liu and H.~Liu.
\newblock Boundary conditions for the microscopic {FENE} models.
\newblock {\em SIAM J. Appl. Math.}, 68(5):1304--1315, 2008.

\bibitem{LiSh09}
H.~Liu and J.~Shin.
\newblock The Cauchy-Dirichlet problem for the {FENE} model of polymeric flows.
\newblock {In Preparation}.

\bibitem{LZ07}
T.~Li and P.~Zhang.
\newblock Mathematical analysis of multi-scale models of complex fluids
\newblock {\em Commun. Math. Sci.}, 5:1--51, 2007.

\bibitem{Ma08}
N.~Masmoudi.
\newblock Well-posedness for the {FENE} dumbbell model of polymeric flows.
\newblock {\em Comm. Pure Appl. Math.}, 61(12):1685--1714, 2008.

\bibitem{Ne62}
J.~Ne{\v{c}}as.
\newblock Sur une m\'ethode pour r\'esoudre les \'equations aux d\'eriv\'ees
  partielles du type elliptique, voisine de la variationnelle.
\newblock {\em Ann. Scuola Norm. Sup. Pisa (3)}, 16:305--326, 1962.

\bibitem{Oe:1996}
H.~{\"O}ttinger.
\newblock {\em Stochastic Processes in Polymeric Liquids}.
\newblock Springer-Verlag, Berlin and New York, 1996.

\bibitem{Rm91}
M.~Renardy.
\newblock An existence theorem for model equations resulting from kinetic
  theories of polymer solutions.
\newblock {\em SIAM J. Math. Anal.}, 22(2):313--327, 1991.

\bibitem{ZhZh06}
H.~Zhang and P.~Zhang.
\newblock Local existence for the {FENE}-dumbbell model of polymeric fluids.
\newblock {\em Arch. Ration. Mech. Anal.}, 181(2):373--400, 2006.





\end{thebibliography}

\end{document}